\documentclass{amsart}
\usepackage{amssymb,latexsym}
\usepackage{graphicx}
\usepackage{url}
\DeclareGraphicsExtensions{.pdf,.png,.jpg,.eps}
\usepackage{epstopdf}
\usepackage{bbm}
\usepackage[utf8]{inputenc}
\usepackage{hyperref}
\usepackage{cleveref}
\crefname{section}{§}{§§}
\Crefname{section}{§}{§§}
\newcommand\C{\mathbb{C}}
\newcommand{\complex}{\mathbb{C}}
\newcommand\N{\mathbb{N}}
\newcommand\Z{\mathbb{Z}}
\newcommand\R{\mathbb{R}}

\newcommand\g{\mathfrak{g}}

\newcommand\Ad{\operatorname{Ad}}
\newcommand\tr{\operatorname{tr}}

\newcommand\re{\operatorname{Re}}

\newtheorem{thm}{Theorem}[section]
\newtheorem{prop}[thm]{Proposition}
\newtheorem{lemma}[thm]{Lemma}
\newtheorem{cor}[thm]{Corollary}

\theoremstyle{definition}
\newtheorem{defn}{Definition}
\theoremstyle{remark}
\newtheorem{remark}{Remark}

\numberwithin{equation}{section}

\begin{document}
\title{Character Estimates of Adjoint Simple Lie Groups}
\author{Corey Manack}
\address{Department of Mathematics\\
  Amherst College\\
  Amherst, MA 01002}
\begin{abstract}
Call a compact, connected, simple Lie group $G$ {\emph{adjoint simple}} if it is also centerless. Let $C\subset G$ be a conjugacy class. We prove the existence of an $n\in\N$, depending on $G$ but not $C$, such that $C^n$ contains a neigborhood of the identity. We then prove that a disk $D\subset\C$, with radius less than $1$, contains the image of every normalized character $\chi(e)^{-1}\chi$ of $G$.
\end{abstract}

\maketitle
\section{Introduction, Statement of Results}

We say that a compact, connected Lie group $G$ is $\textit{simple}$ if its Lie algebra is simple, and $G$ is $\textit{adjoint simple}$ if $G$ is simple and centerless. Let $\g=T_e(G)$ be the Lie algebra of $G$, $(\pi,V)$ a finite-dimensional representation of $G$, $\chi = \tr \pi$ the {\emph{character}} of $\pi$, $e$ the identity element of $G$, $\dim V=\chi(e)$ the {\emph{degree}} of $\pi$. The quantity $\chi(g)/\chi(e)\in\complex$ is the $\textit{normalized character}$ of $V$ at $g\in G$ which is the average of the eigenvalues of $\pi(g)$. By compactness, a normalized character takes values in the closed unit disk in $\complex$. When $G$ is adjoint simple, we can say more.

\begin{thm}
\label{thm:smalldisk}
Let $G$ be adjoint simple. Then there exists a real constant $c$, $-1< c < 0$, depending only on $G$, such that the image of any normalized character is contained in the disk tangent to the line $\re z = c$, and the unit circle.
\end{thm} 
\begin{figure}
\label{chidisk}
\includegraphics{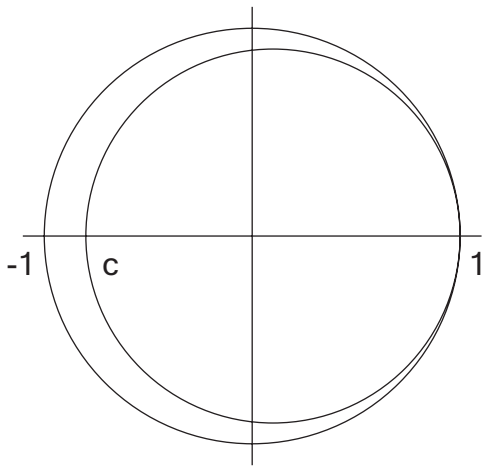}
\caption{Disk of values for $\frac{\chi(g)}{\chi(e)}$.}
\label{Fig:Disk}
\end{figure}
See Figure \ref{Fig:Disk}. The condition that $G$ is adjoint simple is essential. Let $Z(G)$ be the center of $G$. If $G$ is simple then $|Z(G)|$ is finite; if $|Z(G)|>1$, there is a complex irreducible representation $\pi$ and nontrivial $x\in Z(G)$, such that $\pi(x) \neq I$ (\cite{Fulton-Harris:1991} pg. $438$, Prop, $23.13$). As $\pi(x)$ commutes with all $\pi(g)$ and $\pi$ irreducible, $\pi(x) =  \lambda I$ by Schur's Lemma. As $G$ is compact, $|\lambda| = 1$, and since $x$ has finite order, $\lambda$ is a $k$-th root of unity $\neq 1$. Taking traces, $\lambda = \chi(x)/\chi(e)$ but Theorem \ref{thm:smalldisk} implies that the image of $\chi(e)^{-1}\chi$ avoids points on the unit circle, except at $1$. 

The proof of Theorem \ref{thm:smalldisk} relies on the following asymptotic structure theorem for products of $n$ elements from a conjugacy class $C\subset G$. Denote by $C^n$ the set of $n$-fold products from $C$, i.e.,
\[C^n = \{x_1x_2\cdots x_n\mid x_1,x_2,\ldots,x_n\in C\}.\]  
\begin{thm}
\label{thm:Cnopen}
Let $G$ be a compact, adjoint simple Lie group. For $n$ sufficiently large, any nontrivial conjugacy class $C\subset G$ has the property that $C^n$ contains the identity as an interior point.
\end{thm}    
We note that conjugacy classes of a compact analytic manifold are regular submanifolds (\cite{Varadarajan:1974} Corollary 2.9.8). If $C$ is nontrivial then $C$ has positive dimension. Intuitively, adjoint simplicity of $G$ implies the dimension of $C^n$ increases in $n$ and adjoint simplicity excludes undesirable periodic behavior in $n$.
The strategy for proving Theorem \ref{thm:Cnopen} is standard. Section \ref{sec:two} contains a linearized version of Theorem \ref{thm:Cnopen} which states, for $n$ sufficiently large and arbitrary $X$ in $\g$, one can find points $x_1,x_2,\ldots,x_n$ from $G$, such that the function
\begin{equation}
\label{lineargoal}
F(g_1,\ldots,g_n) = \Ad(g_1)X+\Ad(g_2)X+\cdots+\Ad(g_n)X 
\end{equation}
vanishes at $g_1 = x_1, \ldots g_n = x_n$, and is submersive there. In section \ref{sec:three}, we carefully exponentiate \eqref{lineargoal} to prove Theorem \ref{thm:Cnopen} for conjugacy classes in a punctured neighborhood of $e$. The remaining conjugacy classes follow by a submersion argument and compactness of $G$.  Section \ref{sec:four} contains the character estimates. Most of the work is spent proving that the constant $c$ of Theorem \ref{thm:smalldisk} is greater than $-1$. Then $c<0$ from the following observation: for any nontrivial irreducible character $\chi$, orthogonality of $\chi$ to the trivial character implies \[\int_G\re \chi(g)dg = 0.\] Since $\chi(e)>0$, $\re\chi$ must assume a negative value at some non-identity element.

\section{Sums of Adjoint Orbits}
\label{sec:two}
In this section, assume $G$ is simple. We employ tools from convex analysis \cite{Rockafellar:1970}. Throughout, assume $V$ is a real, finite-dimensional vector space. 
\begin{lemma}
\label{thm:halfspace}
Let $V$ be a real vector space and $A$ a compact, convex subset of $V$. Then either
\begin{enumerate} 
\item{ $A$ contains $0$ as an interior point, or}
\item{ $A$ is contained in a closed homogeneous half-space $H$ of $V$, i.e., $v^{*}(A)\geq 0$ for some nonzero dual vector $v^{*}\in V^*$.} 
\end{enumerate}
\end{lemma}
\noindent The proof of Lemma \ref{thm:halfspace} can be found in \cite{Rockafellar:1970}.
\begin{thm}
\label{thm:zeroinconvexhull}
Let $G$ be simple, $(\pi,V)$ a nontrivial irreducible representation of $G$. If $v\in V$ is nonzero, there exists points $g_1,\ldots, g_n$ in $G$ for which $0$ is contained in the interior of the convex hull of $\pi(g_1)v,\pi(g_2)v,\ldots,\pi(g_n)v$.
\end{thm}
\begin{proof}
Let $v\in V\setminus \{0\}$ be given. Suppose towards contradiction that $0$ is not in the interior of the convex hull of $\pi(G)v$. Lemma \ref{thm:halfspace} says $\pi(G)v$ must lie in a homogeneous half space, which we call $H$.  In other words, there exists $v_0^* \in V$ such that $v_0^*(\pi(g)v) \geq 0$ for all $g\in G$. As the $\R$-span of $\pi(G)v$ is a nonzero invariant subspace of $V$ and $V$ is irreducible, $\operatorname{span}_{\R}{\pi(G)v} = V$. Consequently, $\pi(G)v$ is not contained solely in the boundary of $H$, and by continuity of $\pi$, $v_0^*(\pi(u)v)> 0$ for all $u$ in some open subset $U$ of $G$.   
Now, define $f \in V^*$ to be the covector $f(w)=\int_G w^*(\pi(g)v)dg$, where $dg$ is the normalized Haar measure on $G$. Since $dg$ is $G$-invariant, so is $f$. Write $f(v_0)$ into the sum of two integrals
\vspace{3 mm}
\begin{align*}
f(v_0)= \int_U v_0^*(\pi(g)v) dg + \int_{G\backslash U} v_0^*(\pi(g)v) dg,
\end{align*}
noting that the first integral is positive and the second integral is non-negative. Therefore $f(v_0) > 0$, meaning $f$ is a nontrivial $G$-invariant vector of $V^*$ which contradicts irreducibility of $V^*$ ($\dim V < \infty$). Thus, $0$ can be expressed as a positive convex combination of $\pi(g_1)v,\pi(g_2)v,\ldots,\pi(g_n)v$. As $\operatorname{span}_{\R}{\pi(G)v} = V$ and $G$ acts transitively on $\pi(G)v$, the vectors $\pi(g_1)v,\pi(g_2)v,\ldots,\pi(g_n)v$ can be chosen so that its convex hull has nonempty interior, augmenting $n$ if necessary.
\end{proof}
\begin{defn}
A smooth map $f\colon M\to N$ of manifolds is \emph{submersive} at $m\in M$ if the differential $df_m$ is surjective; $f$ is called a submersion if it is submersive at every point $m$.  
\end{defn}
\begin{lemma}
\label{thm:submersive}
Let $G$ be compact and simple, $(\pi,V)$ an irreducible representation of $G$. For fixed $v\in V$, let $f: G \to V$ be the map
$f(x) = \pi(x)v.$ If there are points $x_1,x_2,\ldots,x_n$ such that $f(x_1),\ldots,f(x_n)$ spans $V$, then 
\[df_{x_1}T_{x_1}(G)+df_{x_2}T_{x_2}(G)+\cdots+df_{x_n}T_{x_n}(G) = V.\] 
 
\end{lemma}
\begin{proof}
Suppose $x := (x_1,\ldots,x_m)$ is a point in $G^n$ for which $\{\pi(x_i)v\}_{i=1}^n$ spans $V$. Under the tangent space identifications $T_x(V)\cong V$ and \[T_{(x_1,x_2,\ldots,x_n)}(G^n) \cong T_{x_1}(G)\times\cdots \times T_{x_n}(G),\] the statement \[df_{x_1}T_{x_1}(G)+df_{x_2}T_{x_2}(G)+\cdots+df_{x_n}T_{x_n}(G) = V\] is equivalent to showing that the map 
\[F\colon (g_1,g_2,\ldots,g_n)\to \pi(g_1)v+\pi(g_2)v+\cdots+\pi(g_n)v\] is submersive at $x=(x_1,x_2,\ldots,x_n)$. We show this by taking directional derivatives. For fixed $X\in \g$ and $n$-tuple of scalars $c = (c_1,\ldots,c_n)$ , let $\gamma_X :\R\times\R^n \to G^n$ be the smooth function 
\[\gamma_X(t,c) =(\exp(tc_1X)x_1,\ldots,\exp(tc_nX)x_n).\]
As a function of $t$, $\gamma_X(t,c)$ is a right-translated one parameter subgroup through $x$, with 
\[\gamma_X'(0,c) =(d(r_{x_1})X,d(r_{x_2})X,\ldots,d(r_{x_n})X)\in T_{x_1}(G) \times \cdots\times T_{x_n}(G). \]
where $r_{x_i}$ is the right multiplication map by $x_i$. 
Explicit calculation shows that the image of $dF_x$ contains
\begin{align*}
dF_x(d(r_{x_1})X,d(r_{x_2})X,\ldots,d(r_{x_n})X) &= \left.\frac{d}{dt}\right|_{t = 0}\!\!\!\!\!\!F(\gamma_X(t,c))\\
&= \left.\frac{d}{dt}\right|_{t = 0} \sum_{i=1}^n\pi((\exp tc_iX) x_i)v\\
&= \left.\frac{d}{dt}\right|_{t = 0} \sum_{i=1}^n(\exp tc_id\pi(X))\pi(x_i)v\\
&= d\pi(X)\sum_{i=1}^nc_i\pi(x_i)v.
\end{align*}
where the penultimate equality follows from the intertwining property of $\exp$ (equation \ref{equ:intertwines}), and the last equality follows by applying the derivative to each term, factoring $d\pi(X)$ from the sum. As the choice of $X$, $c$ was arbitrary, the image of $dF_{x}$ contains the subspace 
\[d\pi(\g)\left(\operatorname{span}_{\R}\{\pi(x_1)v,\ldots,\pi(x_n)v\}\right).\]
By assumption, $\operatorname{span}_{\R}\{\pi(x_1)v,\ldots,\pi(x_n)v\} = V$, and $V$ remains irreducible as a representation of $\g$, whence $d\pi(\mathfrak \g)V = V$. The lemma is proved.  
\end{proof} 
\begin{cor}
\label{thm:convexsubmersive}
Let $G$ be a compact Lie group, $(\pi,V)$ an irreducible representation of $G$. For fixed $v\in V$, let $f: G \to V$ be the smooth map
$f(x) = \pi(x)v.$ Suppose the convex hull of $f(x_1),\ldots,f(x_n)$ contains $0$ as an interior point. Then 
\[df_{x_1}T_{x_1}(G)+df_{x_2}T_{x_2}(G)+\cdots+df_{x_n}T_{x_n}(G) = V.\]
\end{cor}
\begin{proof}
If $f(x_1),f(x_2),\ldots,f(x_n)$ spans a proper subspace of $V$, then the convex hull of \[f(x_1),f(x_2),\ldots,f(x_n)\] is a subset of $V$ with empty interior. So $f(x_1),f(x_2),\ldots,f(x_n)$ spans $V$, and Lemma \ref{thm:submersive} applies.
\end{proof} 
Next, we record the version of the Implicit Function Theorem needed for this paper.
\begin{thm}[Local Submersion Theorem]
\label{thm:LST}
Let $M,N$ be smooth manifolds, If $f\colon M\to N$ is a submersion at a point $m\in M$, and $f(m) = n$, then there exist local coordinates around $m$ and $n$ such that $f$ is locally equivalent to a projection map $\R^{k+l}\to\R^l$.
\end{thm}
\begin{cor}
\label{thm:submersiveinnhd}
If $f\colon M\to N$ is a submersion at $m$ then it is a submersion in a neighborhood of $m$.
\end{cor}
\begin{prop}
\label{thm:Rn}
Let $n \geq 1$. For positive reals $a_1,\ldots,a_n$, let $A$ be the ray in $\R^n$ emanating from $0$ and passing through the point $(a_1,\ldots,a_n)$. There exists an infinite sequence of points $x_1,x_2,\ldots$ in $\Z^n$, such that $x_1$ is a coordinate vector, $x_{i+1}-x_i$ is a coordinate vector, and $x_i$ is within distance $\sqrt{2n}$ of $A$, for all $i$.     
\end{prop} 
\begin{proof}
As $a_1,\ldots,a_n$ are positive, the function $f\colon \R^+\to Z^n$ given by \[f(t) = (\left\lfloor ta_1 \right\rfloor,\ldots ,\left\lfloor ta_n \right\rfloor)\] is nondecreasing with respect to the lexicographic ordering $\preceq$. Therefore, the image of $f$ can be listed as a sequence $y_0, y_1,y_2,\ldots$ of points in $\Z^n$ with the properties $y_0 = 0$, $y_{i}\prec y_{i+1}$ for all $i\in\N$. Hence, the difference $y_{i+1}-y_i$ is $\sum_{k=1}^n\delta_{k}e_k,$ where $\delta_{k} = 0$ or $1$ and $\delta_k = 1$ for at least one $k$. We may therefore construct a sequence of lattice points $x_1,x_2,\ldots$, containing $y_1,y_2,\ldots $ as a subsequence, such that $x_1$ is a coordinate vector, and $x_{i+1}-x_i$ a coordinate vector. 
\noindent The definition of the floor function implies, for all nonnegative $t$,
\[\sup_i \left\vert\left\lfloor ta_i \right\rfloor - ta_i\right\vert^2 \leq 1.\]
For those $t$ such that $f(t) = y_k$, we have, by triangle inequality,
\[\left|y_k  - t(a_1,\ldots,a_n)\right|^2 \leq \sup_i n\left\vert\left\lfloor ta_i \right\rfloor - ta_i\right\vert^2 n \leq n\]
and 
\[\left| y_{k+1} - t(a_1,\ldots,a_n)\right|^2 \leq 2n.\]
By construction, for any $x_l$ there exists $y_k$ so that $y_k\preceq x_l\prec y_{k+1}$; it follows from the last estimate that 
\[\inf_{t\in \R^+}\left|x_{l} - t(a_1,\ldots,a_n)\right|^2 \leq 2n\]
The sequence $x_1,x_2\ldots$ satisfies the lemma.
\end{proof}
\begin{lemma}
\label{thm:addvect}
Let $S=\{v_1,\ldots v_n\}$ be a set of vectors in $V$ and $a_1,\ldots,a_n$ positive real numbers such that $\sum_{i=1}^n a_iv_i = 0$. Then there exists $R>0$, and a sequence $w_1,w_2,\ldots$ of vectors from $S$ such that all partial sums $\sum_{i=1}^n w_i$ are bounded by $R$.
\end{lemma}
\begin{proof}
Let $A$ be the ray in $\R^n$ anchored at $0$ and passing through $a = (a_1,\ldots,a_n)$. By Proposition \ref{thm:Rn}, there exists a sequence of points $x_1,x_2,\ldots$ from $\Z^n$, such that $x_i$ is within distance $\sqrt{2n}$ of $A$, and consecutive terms differ by a coordinate vector. Let $ta$ be a point on the ray whose distance to $x_k$ is at most $\sqrt{2n}$; if we write $x_k$ as $(x_{k1},x_{k2}\ldots,x_{kn})$, then by comparing coordinates, 
\[\sup_i\left|x_{ki}-t a_i\right| \leq \left|x_k-ta\right|< \sqrt{2n}.\]
By assumption, $t a_1v_1+\cdots +t a_nv_n = 0$, yielding
\begin{align*}
\left|x_{k1}v_1+\ldots+x_{kn}v_n\right| &= \left|x_{k1}v_1+\cdots+x_{kn}v_n- (t a_1v_1+\cdots +t a_nv_n)\right|\\
                                          &\leq |x_{k1}-ta_1||v_1|+\cdots +|x_{kn}-ta_n||v_n|\\ 
                                          &\leq n\sqrt{2n}\sup_j|v_j|.
\end{align*}
Let $R= n\sqrt{2n}\sup_j|v_j|$ and define $f:\Z^n \to V$ to be the linear map
\[f(b_1,\ldots,b_n) = b_1v_1+b_2v_2+\ldots + b_nv_n.\] 
Notice that $f$ maps the set of coordinate vectors of $\R^n$ onto $S$. By the construction in Proposition \ref{thm:Rn}, $f(x_1)\in S$, and since $x_{i+1}-x_i$ is a coordinate vector, $f(x_{i+1}-x_i)\in S$. Explicitly, the sequence
\[f(x_1),f(x_2-x_1),f(x_3-x_2),\ldots\]
has $k$-th partial sum $f(x_k)$, which was just shown to satisfy $|f(x_k)|\leq R$.
\end{proof}
\begin{thm}
\label{thm:linearized}
Let $V$ be a finite dimensional real vector space,
$M$ a smooth manifold, $f\colon M\to V$ a smooth function, $x_1,\ldots,x_m$
points of $M$ such that $0$ lies in the interior of the convex hull of $f(x_1),\ldots,f(x_m)$,
and 
\begin{equation}
\label{equ:subcond}
V = df_{x_1}T_{x_1}(M) + \cdots + df_{x_m}T_{x_m}(M).
\end{equation}
Then, for all $n$ sufficiently large, there exist $z_1,\ldots,z_n\in M$ such that
\begin{equation}
\label{equ:vanwant}
\sum_{i=1}^n f(z_i) = 0,
\end{equation} and 
\begin{equation}
\label{equ:subwant}
V = df_{z_1}T_{z_1}(M) + \cdots + df_{z_n}T_{z_n}(M).
\end{equation}
\end{thm}
\begin{proof}
Since $0$ lies in the interior of the convex hull of $f(x_1),\ldots,f(x_m)$, for $\epsilon >0$ sufficiently small, we can find convex coefficients $a_1,a_2,\ldots,a_m$ such that 
\[a_1f(x_1) + a_2f(x_2) + \cdots + a_mf(x_m) = -\epsilon(f(x_1)+\cdots+f(x_m)).\]
Solving and renormalizing, $f(x_1),\ldots,f(x_m)$ satisfy 
\begin{equation}
\label{equ:convanish}
a_1f(x_1) + a_2f(x_2) + \cdots + a_mf(x_m) = 0
\end{equation} 
and each convex coefficient $a_1,a_2,\ldots,a_m$ in \eqref{equ:convanish} is positive. Keeping $m$ fixed but allowing $k\in\N$ to vary, write $a \in\R^m$ as $a=(a_1,\ldots,a_m)$, $x\in M^m$ as $x=(x_1,\ldots,x_m)$, $t\in \R^k$ as $t = (t_1,\ldots,t_k)$, $y\in M^k$ as $y = (y_1,\ldots,y_k)$, and let $\mathbbm{1}_k=(1,\ldots,1)$  ($k$ times). Define the family of smooth maps $F_1,F_2,\ldots$ to be
\[F_k(t,y) = \sum_{i=1}^kt_if(y_i),\]
and let $F_k^t\colon M^k\to V$ be the restriction of $F_k$ to the slice $\{t\}\times M^k$. If every component of $t$ is nonzero, the image of the differential of $F_k^t$ at $y$ is the subspace 
\[df_{y_1}T_{y_1}(M) + \cdots + df_{y_k}T_{y_k}(M).\]
So the theorem is proved once we produce an $N\in\N$ such that, for all $n\geq N$ some point $z=(z_1,\ldots,z_n)$ in $M^n$ satisfies $F_n^{\mathbbm{1}_n}(z)=0$ and $F_n^{\mathbbm{1}_n}$ is submersive at $z$, because the coordinates of such a $z$ would satisfy properties \eqref{equ:vanwant} and \eqref{equ:subwant} simultaneously. At this point, we know $F_m^a(x)=F_m(a,x)=0$ and $F_m^a$ is submersive at $x$. The Implicit Function Theorem grants a $\delta>0$ and a smooth function \[g\colon a+(-\delta,\delta)^m \to M^m\] satisfying $g(a) = x$, $F_m^t(g(t)) = F_m(t,g(t)) = 0$ for $t$ satisfying $\sup_i|t_i-a_i|<\delta$. 
\noindent By Corollary \ref{thm:submersiveinnhd} and our spanning assumption \eqref{equ:subcond}, $F_m^t$ is a submersion at all points in a neighborhood $U$ of $x$, so long as each component of $t$ is nonzero.  Choosing $\delta$ small enough so that each $t\in  a+(-\delta,\delta)^m$ has positive coordinates and $g(a+(-\delta,\delta)^m)\subset U$, we may assume that $F_m^t$ is a submersion at $y=g(t)$ for all $t\in  a+(-\delta,\delta)^m$. In particular, there exists an $m$-tuple of positive rationals \[\frac{p}{q}:=\left(\frac{p_1}{q_1},\frac{p_2}{q_2},\ldots,\frac{p_m}{q_m}\right)\] such that $p_i/q_i$ falls within $\delta$ of $a_i$. It follows that 
$F_m^{p/q}$ vanishes at the point 
\[c:= (c_1,\ldots,c_m) := g\left( \frac{p_1}{q_1},\frac{p_2}{q_2},\ldots,\frac{p_m}{q_m}\right)\in X^m\] and $F_m^{p/q}$ is submersive at $c$. 
Denote by $\Delta^k$ the diagonal submanifold of $M^k$; if we let 
\begin{equation}
\label{existn}
n= p_1q_2\cdots q_m + q_1p_2\cdots q_m + \cdots + q_1q_2\cdots p_m,
\end{equation}
then notice that the function $q_1q_2\cdots q_mF_m^{p/q}$ is simply the restriction of $F_n^{\mathbbm{1}_n}$ to the product of diagonal manifolds \[\Delta^{p_1q_2\cdots q_m}\times\Delta^{q_1p_2\cdots q_m}\times\cdots\times\Delta^{q_1q_2\cdots p_m}\subset M^{n}.\] Therefore, the image of the differential of $q_1q_2\cdots q_mF_m^{p/q}$ at $c$ is a subspace of the image of the differential of $F_n^{\mathbbm{1}_n}$ at  
\begin{equation}
\label{npoints}
z=(z_1,\ldots,z_n):=(\underbrace{c_1,\ldots,c_1}_{p_1q_2\cdots q_m},\underbrace{c_2,\ldots,c_2}_{q_1p_2\cdots q_m},\ldots,\underbrace{c_m,\ldots,c_m}_{q_1q_2\cdots p_m}).
\end{equation}
Since $q_1q_2\cdots q_mF_m^{p/q}(c) = 0$ and $q_1q_2\cdots q_mF_m^{p/q}$ is submersive at $c$, we've shown the existence of an $n$, given by equation \eqref{existn} and $n$-tuple $z=(z_1,\ldots,z_n)$, defined in \eqref{npoints}, such that $F_n^{\mathbbm{1}_n}(z) = 0$ and $F_n^{\mathbbm{1}_{n}}$ is submersive at $z$. The theorem would follow once \eqref{equ:vanwant} and \eqref{equ:subwant} are satisfied for all integers beyond a sufficiently large multiple of $n$. Consider $F_n^{\mathbbm{1}_n}$ near $z$. By Theorem \ref{thm:submersiveinnhd}, there exists an open set $Z\subset M^n$ containing $z$ and an open ball $B_r(0)\subset V$ such that $F_n^{\mathbbm{1}_n}$ is a submersion on $Z$ and $F_n^{\mathbbm{1}_n}( Z)=B_r(0)$. Notice further, for all $k,l\in\N$, the function $F_{ln+k}^{\mathbbm{1}_{ln+k}}$ remains a submersion at all points in $Z^l\times X^k$. By linearity,  
\begin{equation}
\label{equ:ontoball}
F_{ln+k}^{\mathbbm{1}_{ln+k}}(Z^l\times \{y\})=B_{lr}(0)+\sum_{i=1}^kf(y_i).
\end{equation}   Now, recalling that $0$ lies in the interior of the convex hull of $f(x_1),f(x_2),\ldots,f(x_m)$, by Lemma \ref{thm:addvect}, there is a bound $R>0$ and sequence of vectors $f(x_{i_1}),f(x_{i_2})\ldots$, where $x_{i_j}\in\{x_1,x_2,\ldots,x_m\}$, such that the sequence of vectors \[s_k:=\sum_{j=1}^k f(x_{i_j})=F_k(\mathbbm{1}_k,x_{i_1},\ldots,x_{i_k})\] satisfy $s_k\in B_R(0)$ for all $k\in\N$. Consequently, if $l\in\N$  is chosen to satisfy $lr\geq R$, then the translated ball $B_{lr}(0)+s_k\subset V$ contains $0$, for all $k\in\N$. By \eqref{equ:ontoball}, there exists a point $z^*\in Z^l\times \{x_1,\ldots,x_m\}^k$ such that $F_{ln+k}^{\mathbbm{1}_{ln+k}}(z^*) = 0$ and $F_{ln+k}^{\mathbbm{1}_{ln+k}}$ is submersive at $z^*$. Therefore, if $n'\geq ln$, there exists $z^*\in M^{n'}$ whose coordinates satisfy \eqref{equ:vanwant} and \eqref{equ:subwant}. The Theorem is proved.

\end{proof}
Let $G$ be simple. For fixed $X\in\g$, let $f(x) = \Ad(x)X$. Viewing $\g$ as a real vector space on which $\Ad$ acts irreducibly (\cite{Fulton-Harris:1991}, page 434),  Lemmas \ref{thm:zeroinconvexhull}, \ref{thm:convexsubmersive} specialize to yield $0$ as expressed by a sufficiently large {\emph{convex}} combination of vectors $\Ad(x_1)X,\ldots,Ad(x_n)X$, and the function \[(g_1,\ldots,g_n)\to \Ad (g_1)X+\cdots + \Ad (g_n)X\] is submersive at $(x_1,\ldots,x_n)$. Theorem \ref{thm:linearized} shows that the convex coefficients $a_1,\ldots,a_n$ can be replaced by sufficiently large positve integer coefficients $b_1\ldots,b_n$, retaining the vanishing and submersivity properties of Corollary \ref{thm:convexsubmersive}.   In other words, for sufficiently large $n$, the sum of $n$ copies of $\Ad(G)X$ contains a neighborhood of $0$. Notice that $n$ may depend on $X$. We lift this dependency, at the group level, in the next section. 
\section{Powers of a Conjugacy Class} 
\label{sec:three}
We record the three properties of the exponential map $\exp\colon \g\to G$ used in this section \cite{Varadarajan:1974}: $t \to \exp tX$ is a homomorphism from $\R$ to $G$, $\exp$ is a local diffeomorphism near $0\in\g$ whose local inverse we denote by $\log$, and if $(\pi,V)$ is any representation of $G$, $\exp$ intertwines the $\pi$ with its differential $d\pi$, i.e.,
\begin{equation}
\label{equ:intertwines}
\pi \circ \exp = \exp \circ\ \! d\pi,
\end{equation}
for all $X\in\g$. The fourth property, Theorem \ref{BCH}, is the Baker--Campbell--Hausdorff (BCH) formula, which expresses the product of $n$ elements near $e$ in terms of their exponentials. The next results are developed in \cite[Section 2.15]{Varadarajan:1974} for $n=2$ and can be extended to arbitrary $n$ by induction. 
\begin{thm}[\cite{Varadarajan:1974}, Theorem 2.12.4, Theorem 2.15.4, equation (2.15.5)]
\label{BCH} Let $G$ be a Lie group, $\g$ the Lie algebra of $G$, $S(\g)$ the unit sphere in $\g$ with respect to the Killing form, $B(\g)$ the unit ball, $n\geq 1$, and $X_1,\ldots,X_n \in \g$. There exists a  unique smooth function $r_n\colon \g^n\to \g$, such that, for $\vert t\vert$ sufficiently small, 
\[\exp tX_1\cdots\exp tX_n = \exp(t(X_1+\cdots+X_n) + r_n(tX_1,\ldots,tX_n)).\]
Moreover, for $l=1,2$ consider the two functions
\begin{equation}
\label{equ:smoothrem}
 R_n(t,X_1,\ldots,X_n)= \begin{cases}
           \displaystyle t^{-l}r_n(tX_1,\ldots,tX_n), & t\neq 0, \\
            \quad \; 0 & t=0
\end{cases}
 \end{equation} $R_n$ is smooth when $l=1$, and $R_n$ is bounded (with respect to the Killing form) when $l=2$, $\vert t\vert$ sufficiently small and $X_1,\ldots,X_n\in S$.
\end{thm}
\begin{remark}
The remainder function $r_n$ is a power series in the commutators of  $X_1,\ldots,X_n$, given in closed form by Dynkin's formula \cite[remark 3]{Varadarajan:1974}. All we need are the properties of $r_n$ stated in equation \eqref{equ:smoothrem}.
\end{remark}
Define $w_n:G^n\times \g\to G$ to be the {\emph{word map}}  \[w_n(g_1,g_2,\ldots,g_n,X) = g_1\exp X g_1^{-1}g_2\exp X g_2^{-1}\cdots g_n\exp X g_n^{-1}.\]
The image of $w_n$ is the set of $n$-fold products of the conjugacy class through $\exp X$; we would like to examine the image of $\log w_n$. To do so requires that the image of $w_n$ be contained in the domain of $\log$.

\begin{prop}
\label{thm:prodinnhd}
Let $G$ be a compact connected Lie group, $S=S(\g)$ the unit sphere in $\g$ with respect to the Killing form $B=B(\g)$ the unit ball. For all $n>0$, there exists $\delta$ sufficiently small and $\mu>0$ such that for any $1\leq k \leq n$ and $0<t<\delta$ the product of $k$ elements from $\exp(tS)$ is contained in $\exp (t\mu B)$.
\end{prop}
\begin{proof}
Let $n>0$ be given, $R$ the the injectivity radius of $G$, $\delta$ small enough so that $\exp(\delta B)^k\subset \exp RB$, for all $1\leq k \leq n$. Let $1\leq k \leq n$, $0<t<\delta$, $X_1,\ldots, X_k\in S(\g)$; by BCH, \[\log(\exp(t X_1)\exp(t X_2)\cdots\exp(t X_k))=t X_1 +\cdots +t X_k + r_k(t X_1,\ldots,t X_k),\]
and by equation \eqref{equ:smoothrem} (case $l=2$), there exists a bound $m_k>0$ such that $|r_k(t X_1,\ldots,t X_k)| < t^2 m_k$ for all $X_1\ldots,X_k\in S(\g)$.  Consequently, 
\[\left|\log(\exp(tX_1)\exp(t X_2)\cdots\exp(t X_k))\right| \leq t k + t^2m_k.\]
Choosing $\mu= \max\{k +\delta m_k\mid 1\leq k\leq n\}$, observe
\[tk+ t^2m_k < t(k + \delta m_k) \leq t\mu\]
implies
\[\log(\exp(tX_1)\exp(tX_2)\cdots\exp(tX_k)) \in t\mu B\]
for all $1\leq k \leq n$, $0<t<\delta$, $X_1,\ldots, X_k\in S.$ The proof follows by exponentiation.
\end{proof}
The next lemma will provide sufficient control over the $r_n$ term of the BCH formula to prove Theorem \ref{thm:Cnopennhd} for sufficienly large multiples of conjugacy classes $C$ near $e$. Then Proposition \ref{thm:prodinnhd} is used to fill in intermediate powers of $C$. Abusing notation, write $g(u,V)$ for the image $g(\{u\}\times V)$.
\begin{lemma}
\label{thm:betterfact}
Let $M,N,P$ be smooth manifolds, $g:M\times N \to P$ a smooth function, $g^a$ the restriction of $g$ to the slice $\{a\}\times N$. Suppose $g(a_0,b_0)= c_0$ and $g^{a_0}$ is a submersion at $b_0$. Then there is an open set $U\subset M$ containing $a_0$, an open set $V\subset N$ containing $b_0$, an open set $W\subset P$ containing $c_0$, such that $W\subset g(u,V)$ for all $u\in U$.\\
\end{lemma}
\begin{proof}
We construct, from $g$, a smooth function so as to apply the local submersion theorem. Define $f\colon M\times N\to M\times P$ to be
\[f(m,n) = (m,g(m,n)).\]
As $df = id\times dg$, the image of $dg$ at $(a_0,b_0)$ includes the image of $d(g^{a_0})$ at $b_0$ and $g^{a_0}$ submersive at $b_0$, $f$ is submersive at $(a_0,b_0)$.  By Theorem \ref{thm:LST}, $f$ maps some open neighborhood $U\times V$ of $(a_0,b_0)$ onto an open neighborhood $Q$ of $(a_0,c_0)$. $Q$ contains a neighborhood of the form $U_1\times W$. Comparing second coordinates, we see $g(u,V)$ contains $W$ for all $u\in U_1$.
\end{proof}

\begin{defn}
Let $G$ be a Lie group, $C$ a conjugacy class of $G$, $S$ a subset of $G$. We say that $C$ is $\textit{represented}$ in $S$ if $S\cap C\neq \emptyset$, i.e. $C$ is represented by some $x\in S$. 
\end{defn}
The next Theorem proves Theorem \ref{thm:Cnopen} for conjugacy classes sufficiently close to $e$.  
\begin{thm}
\label{thm:Cnopennhd}
Let $G$ be an adjoint simple Lie group. There exists a punctured neighborhood $U_e^*$ of the identity of $G$, such that, if $C$ is a conjugacy class represented in $U_e^*$, then, for all $n$ sufficiently large, $C^n$ contains $e$ as an interior point.
\end{thm} 
\begin{proof}
Denote by $S$ the unit sphere in $\g$ with respect to the Killing form, $B$ the unit ball in $\g$. As $G$ is adjoint simple, $\g$ is a real vector space on which the adjoint representation of $G$ acts irreducibly. For fixed $X\in S$, let $f(g) = \Ad(g)X$; Theorem \ref{thm:zeroinconvexhull}, specialized to $\pi = \Ad, v=X$, asserts the existence of points $x_1,\ldots,x_m\in G$, for which zero lies in the interior of the convex hull of \[f(x_1), f(x_2) \ldots, f(x_m),\]  so that Corollary \ref{thm:convexsubmersive} implies \[df_{x_1}T_{x_1}(G) + \cdots +  df_{x_m}T_{x_m}(G)= V.\]  
Write $g= (g_1,\ldots,g_n)$ and let $F_1,F_2\ldots$ be the family of maps where $F_n:S\times G^n \to \g$ is given by \[F_n(Y, g) =  \Ad(g_1)Y+\Ad(g_2)Y+\cdots+\Ad(g_n)Y,\] $F_n^Y$ the restriction of $F_n$ to the slice $\{Y\}\times G^n$. As $f$ satisfies the convexity and submersivity conditions in the supposition of Theorem \ref{thm:linearized}, then for $n$ sufficiently large, there is a point $z= z = (z_1,\ldots,z_n) \in G^n$ satisfying 

\[F_n(X,z) = \Ad(z_1)X+\Ad(z_2)X+\cdots+\Ad(z_n)X = 0,\] 
and $F_n^{X}$ is submersive at $x$. Fix one such $n,z$ pair, and let $R>0$ is the injectivity radius of $G$, $\delta$ small enough so that $\exp(\delta B)^k\subset \exp RB$, for all $1\leq k \leq n$. Since spheres in $\g$ with respect to the Killing form are $\Ad$-invariant, $\exp(\delta \Ad(G)B)^k\subset \exp (RB)$, for all $1\leq k \leq n$.
Recall the remainder function $r_n$ of the BCH formula (\ref{BCH}); consider the function $h: (-\delta,\delta)\times S \times G^n\to \g$ is given by
\begin{equation}
 h(t,Y,g)= \begin{cases}
           \displaystyle t^{-1}(F_n(tY, g)+r_n(t\Ad(g_1)Y,\ldots,t\Ad(g_n)Y)), & t\neq 0, \\
            \quad \;  F_n(Y,g) & t=0
\end{cases}
 \end{equation}
By equation \eqref{equ:smoothrem} (case $l=1$), $h$ is smooth for all $\vert t \vert$ sufficiently small. Further, $h(0,X,z) = F_n^{X}(z)=0$. Apply Lemma \ref{thm:betterfact} to the special case $M = (-\delta,\delta)\times S$, $a_0 = (0,X)$, $N=G^n$, $b_0=x$, $P=\g$, $c_0 = 0$ to assert the existence of $\epsilon>0$, an open set $U_X\subset S$ containing $X$, an open set $V\subset G^n$ containing $z$, and an open ball $\rho B\subset \g$, such that 
\begin{equation}
\label{ball}
\rho B\subset h(t,Y,V), 
\end{equation} for all $Y\in U_X$, $t\in (-\epsilon,\epsilon)$. Thus, for $t\in (0,\epsilon)$ and $Y\in U_X$, BCH and the intertwining property of equation \eqref{equ:intertwines} shows
\begin{align*}
\label{param}
& \exp th(t,Y,g)\\ =& \exp(t\Ad(g_1)Y)\exp( t\Ad(g_2)Y)\cdots\exp( t\Ad(g_n)Y) \notag \\
	=& g_1\exp(tY)g_1^{-1}g_2\exp(tY)g_2^{-1}\cdots g_n\exp(tY))g_n^{-1} \notag \\
	=& w_n(g_1,\ldots,g_n, tY)\notag
\end{align*}
Notice, if $C$ is the conjugacy class $C=\exp(t\Ad(G)Y)$, then $\exp th(t,Y,G^n) = C^n$. Consequently, for all $t\in (0,\epsilon), Y\in U_X\subset S$, if $C=\exp (t\Ad(G)Y)$, then $C^n$ contains the open neighborhood $\exp t\rho B$ of $e$. It follows from
\[\exp(bB)\subset \exp(B)^b\]
that $C^{bn}$ contains $\exp tb\rho B$ for all positive integers $b$. Since $t<\epsilon<\delta$ and $\delta$ satisfies the supposition of Proposition \ref{thm:prodinnhd}, there exists $\mu >0$ such that $C^i\subset \exp t\mu B$ for all $1\leq i\leq n$. Selecting $b_0$ to satisfy $b_0\rho \geq \mu$, it follows that $e\in \operatorname{Int}(C^i)$ for all $i\geq b_0n$. Thus, for any conjugacy class $C$ represented in the exponentiated cone $\{\exp(tY)\mid t\in(0,\epsilon), Y\in U_X\}$, $e\in \operatorname{Int}(C^i)$ for all $i\geq b_0n$.  Repeating the above argument at each point $X\in S$ yields a cover of $S$ by neighborhoods $U_X$ satisfying equation \ref{ball}. By compactness of $S$, we may produce the desired punctured neigborhood of $e$ as a finite union of exponentiated cones $\mathcal{C}_1,\ldots,\mathcal{C}_l$, taking the maximum necessary exponent over this union.  
\end{proof}
  
\begin{proof}[Proof of Theorem \ref{thm:Cnopen}]
Let $G$ be adjoint simple; by Theorem \ref{thm:Cnopennhd}, there exists a punctured open neighborhood $U = U_e^*$ of the identity, and sufficiently large $m$ such that, for any conjugacy class $C$ represented in $U_e^*$, $C^m$ contains $e$ as an interior point. We will complete the proof by proving the following statement: for $N$ sufficiently large, $C^N = G$ for any conjugacy class $C$ represented in $G \setminus U_e$.  

For $C\neq \{e\}$ and points $x_1,x_2,\ldots,x_n\in C$, consider the map $u_n\colon G^n\to G$ given by
\[u_n(g_1,\ldots,g_n) = g_1x_1g_1^{-1}g_2x_2g_2^{-1}\cdots g_nx_ng_n^{-1}x_n^{-1}\cdots x_1^{-1}.\]  The image of the differential of $u_n$ at $(e,\ldots,e)$ is the subspace  
\begin{align}
\label{Tanspace}
L_n(x_1,\ldots,x_n):= (1-\Ad(x_1))(\g)&+\Ad(x_1)(1-\Ad(x_2))(\g)+\cdots \\ 
&+\Ad(x_1\cdots x_{n-1})(1-\Ad(x_n))(\g)\notag.
\end{align}
As \[L_n(x) = (1-\Ad(x_1))(\g) + \Ad(x_1)L_{n-1}(x_2,\ldots,x_n),\]
we have $\dim L_n(x_1,\ldots,x_n) \geq \dim L_{n-1}(x_2,\ldots,x_n)$ with equality if and only if \[(\Ad(x_1^{-1})-1)(\g) \subset L_{n-1}(x_2,\ldots,x_n).\]
Notice that $S = \operatorname{span}_{c\in C}(\Ad(c^{-1})-1)(\g)$ is invariant under the adjoint action of $G$. As $c$ is not in the center of $G$, $S\neq {0}$. Since the adjoint representation of $G$ is irreducible, $S = \g$. Hence we can always increase the dimension of $L_{n+1}$, by prepending a word in $C^n$ by a suitable element in $C$, unless the dimension of $L_n$ is already $\dim(\g)$. 
Consequently, for each $x_1\in C$ there exist points $x_2,\ldots,x_n\in C$, $n=\dim G$ so that $u_n$ is submersive at $(e,\ldots,e)$. Lemma \ref{thm:betterfact}, specialized to $M=N=G^n, P=G, m_0=(e,\ldots,e), n_0=(x_1,x_2,\ldots,x_n), p_0=e$, and the smooth map $W_n:G^n\times G^n\to G$ given by
\[W_n(g_1,\ldots,g_n,y_1,y_2,\ldots,y_n) = g_1y_1g_1^{-1}g_2y_2g_2^{-1}\cdots g_ny_ng_n^{-1}y_n^{-1}\cdots y^{-1}\]
produces neighborhoods $U_i$ of $x_i$, $1\leq i \leq n$, and an open neighborhood $V_e$ of $e$, for which \[V_e\subset W_n(G^n,y_1,y_2,\ldots,y_n),\] for all $y_i\in U_i$. As $x_1,x_2,\ldots,x_n\in C$, there exist $h_2,\ldots,h_n\in G$ such that $x_i = h_ix_1h_i^{-1}$, and since the conjugation action of $G$ on $C$ is a smooth and invertible, some neighborhood $V$ of $x_1$ satisfies $V\subset U$ and $h_iVh_i^{-1}\subset U_i$, $2\leq i\leq n$.  
Thus, for any conjugacy class $C'$ represented in $V$, there are $n$ elements $y_1, y_2, \ldots,y_n$ from $C'$ for which \[V_e\subset (C')^ny_n^{-1}\cdots y_2^{-1}y_1^{-1}.\] By compactness of $G$, there exists $m\in\N$ such that $V_e^m=G$, so that
\[G=((C')^ny_n^{-1}\cdots y_2^{-1}y_1^{-1})^m\]
Since $yC'=C'y, yG=G$ for all $y\in G$, $(C')^{mn} = G.$ 
Repeating the above construction for each $x$ in the compact set $G\setminus U_e$ and taking the maximum necessary exponent over a finite subcover proves the Theorem.         
\end{proof}
\begin{remark}
The argument for submersivity of the word map $u_n(g_1,\ldots,g_n)$ at $(e,\ldots,e)$ is similar to an argument in the proof of Proposition 3.2 of \cite{Larsen-Shalev:2009}, where it was used to prove smoothness of $u_n$ for a larger class of groups.
\end{remark}
\begin{remark}
Optimizing the smallest exponent in Theorem \ref{thm:Cnopen} is desired, but not necessary to proceed.
\end{remark}
\begin{remark}
If $G$ is compact, connected, simple but not adjoint simple, then Theorem \ref{thm:Cnopen} is false. Consider a conjugacy class $xC$, $x\in Z(G)$, $C = \exp(\Ad (G)X)$, $X\in\g$. $Z(G)$ has finite order, say $m$ \cite{Fulton-Harris:1991}. For any positive integer $n$, one can find a pair $X,k$ where $X$ is sufficiently near $0$ in $\g$, $k\in\N$ greater than $n$ but coprime to $|x|$ where $(xC)^k$ is contained in an open neighborhood around some element of $Z(G)$ that does not intersect $e$.    
\end{remark}
\begin{remark}
Products of conjugacy classes $C_1,C_2\cdots, C_k$ ought to be suitable for investigation using the methods of Theorems \ref{thm:Cnopen} and \ref{thm:linearized}. 
\end{remark}

\section{Character Estimates}
\label{sec:four}
\begin{lemma}
\label{thm:closedarc}
Let $G$ be adjoint simple, $A$ a closed arc of the unit circle in $\complex$. If $1\notin A$, then there exists $\epsilon >0$  that satisfies the following property: For all characters $\chi$ of finite degree and all $g\in G$, if 
\begin{equation}
\label{sector}
\frac{\chi(g)}{\chi(e)} = (1-\delta)\omega 
\end{equation}
and $\omega\in A$, $0\leq \delta \leq 1$, then  $\delta \geq \epsilon$.
     
\end{lemma}

\begin{proof}
Let $G$ be adjoint simple. By Theorem \ref{thm:Cnopen}, there exists a bound $B>0$ such that $e\in C^k$ for all nontrivial conjugacy classes $C$ and all $k\geq B$. Let $A$ be a closed arc of the circle with $1\notin A$. Treating $A$ as a closed interval in $(0,1)$, equip $\R/\Z$ with the distance induced from the nearest integer norm $\|x\| := \left\lfloor x+1/2\right\rfloor$ on $\R$. Let $m\in [0,1/2]$ be the distance from $A+\Z$ to $\Z$.  As $A$ is a closed interval and $(A+\Z)\cap\Z$ is empty, $m>0$. Pick $q\in\N$ to satisfy $1/q<m$ and $q>B$. Obviously, $q \geq 2$, so there exists $\delta >0 $ such that every $x$ within $\delta$ of the set \[S = \{t/s, 1\leq t < s\leq q\},\] satisfies $jx \mod 1 \in [1/4,3/4]$ for some $q\leq j\leq 2q$. Now let $p\in \N$ satisfy $1/p<\delta$, and let $x\in A$ be arbitrary. Consider the sequence \[x,2x,\ldots,(q+1)x \mod 1.\] Partitioning $[0,1]$ into $q$ intervals of length $1/q$, by the pigeonhole principle, some pair of points lie in the same interval. Subtracting, we see that there is some integer $k$ between $1$ and $q$ for which $\|kx\|\leq 1/q$, i.e., $kx$ is within distance $1/q$ of its nearest integer.  Recalling $1/q<m$ and $x\in A$ shows $k=1$ is impossible.
Similarly, some term of the sequence 
\[kx,2kx,\ldots pkN \mod 1\] satisfies $\|Lkx\|\leq 1/p$. As $1/p < \delta$,  $|Lkx-n| < \delta$ for some $n\in\Z$, or \[|Lx-n/k|< \delta/k<\delta.\]  
Thus, for every point $x$ in $A$, some term in the sequence
\[x,2x,\ldots, (p+1)x \mod 1\] lies within $\delta$ of $S+\Z$. Thus, if $\omega=e^{2\pi ix}$ lies in $A$, then $\re \omega^k \leq 0$ for some $k$ between $B$ and $2pq$. We claim that $\epsilon = 1/(2pq)^2$ satisfies the statement of the theorem. To this end, let  $(\pi,V)$ be any finite dimensional representation of $G$, $g\in G$ arbitrary. Write $\chi= \tr \pi$ for the character of $\pi$, $C$ for the conjugacy class containing $y$ and $n= \dim V = \chi(e)$ . Suppose 
\begin{equation}
\label{equ: chiis}
\chi(g)= n(1-\delta)\omega
\end{equation}
where $\omega\in A$, $1\geq \delta \geq 0$. As $\omega\in A$, $g\neq e$. By Theorem \ref{thm:Cnopen} and the above derivation, there exists an integer $k$, $B\leq k\leq 2pq$ and points $y_1,y_2,\ldots,y_{k} \in C$ that satisfy  
\begin{equation}
\label{equ:prodcond}
y_1\cdots y_k = e.
\end{equation} and $\re \omega^k \leq 0$. In particular, 
\begin{equation}
\label{equ:omest}
\left|\omega^k-1\right|\geq \sqrt{2}.
\end{equation} 
Write $P_i = \pi(y_i)$; each $P_i$ is unitary and any $P_i,P_j$ are unitarily equivalent. If $\lambda_1,\ldots,\lambda_n$ are the eigenvalues of $P_i$ then 
\[\lambda_1+\cdots+\lambda_n = n\omega(1-\delta),\]
where $\omega$ and the $\lambda_i$ lie on the unit circle, so that
\begin{equation}
\label{equ:eigest}
\sum_{i=1}^n|\omega-\lambda_i| = \sum_{i=1}^n(2-\overline{\omega}\lambda_i +\omega\overline{\lambda_i}) = 2n\delta.
\end{equation}
Consider the Hermitian form $\tr A^*B$ on $M_n(\complex)$ which induces the so-called Frobenius norm \[\|A\|:=\sqrt{\tr(A^{*}A)}.\] This norm satisfies the triangle inequality and is unitarily invariant, i.e., $\|UA\|=\|AU\|=\|A\|$ for all $U\in U_n(\complex)$.  Diagonalizing $P_i$ shows, by equation \eqref{equ:eigest}, that \[\|P_i-\omega I\| = \sqrt{2n\delta}\] for each $P_i$.
Writing $P_1\cdots P_k-\omega I$ as the telescoping sum
\[  P_1\cdots P_k- \omega P_1\cdots P_{k-1}+\omega P_1\cdots P_{k-1}+\cdots-\omega^{k-1}P_1+ \omega^{k-1}P_1 -\omega^k I,\]
we get
\begin{align*}
&\| P_1\cdots P_k-\omega^k I\|\notag \\
= &\| P_1\cdots P_k- \omega P_1\cdots P_{k-1}+\omega P_1\cdots P_{k-1}+\cdots-\omega^{k-1}P_1+ \omega^{k-1}P_1 -\omega^k I\|\notag \\
\leq & \| P_1\cdots P_k- \omega P_1\cdots P_{k-1}\| +\cdots + \|\omega^{k-1}P_1 -\omega^k I\|\notag \\
= & \sum_{i=1}^k\|P_i-\omega I\| 
\end{align*}
or $\| P_1\cdots P_k-\omega^k I\|\leq  k\sqrt{2n\delta}$. 
By equation (\ref{equ:prodcond}), 
\begin{equation}
\label{equ:prodcond2}
P_1 \cdots P_{k} = I. 
\end{equation} 
and by estimate \eqref{equ:omest},  \[ \sqrt{2n}\leq \sqrt{n}\left|\omega^k-1\right|=\| P_1\cdots P_k-\omega^k I\| ,\] whence $1/k^2\leq \delta$. As $k\leq 2pq$, $\delta \geq 1/(2pq)^2=\epsilon$.
\end{proof}
\noindent We are now ready to prove the main result.
\begin{proof}[Proof of Theorem \ref{thm:smalldisk}]
For $0< c < 1$, let $\Theta_c\subset \complex$ be the disk of radius $1-c$ centered at $c+0i$ . Proceeding by contradiction, suppose that for every $c$,  there is a pair $x,\chi$ consisting of an element $x\in G$ and a character $\chi$ of $G$, such that $\chi(x)/\chi(e)$ lies outside $\Theta_c$. Any positive sequence $c_1,c_2\ldots$ that tends to $0$ generates a list of complex numbers $z_1,z_2,\ldots$ coming from normalized character values, such that $z_j$ lies in the crescent \[\Theta_{0}\setminus\Theta_{c_j}= \{z\in\complex \mid |z|\leq1, z\notin \Theta_{c_j}\}.\]  Such a sequence satisfies $z_j \neq 1\ \forall j$ but $|z_j|\to 1$. By compactness, some subsequence of $\{z_j\}$ converges to a point on the unit circle; if this point is not equal to $1$, then the limit point is contained in a closed arc of $S^1$ not containing $1$, contradicting Lemma \ref{thm:closedarc}. So, the only possibility is $z_j\to 1$. 
Writing $z_j$ as $r_ke^{i\theta_j}$ where $-\pi\leq\theta_j<\pi$, then $z_j\to 1$ implies $\theta_j\to 0$ and $z_j\in\Theta_{0}\setminus\Theta_{c_j}$ means $\theta_j\neq 0$. Now, fix $0<c<1$ and parametrize the circular boundary of $\Theta_c$ as
\[\Theta_c(\theta) = (1-c)\cos\theta+c + i(1-c)\sin\theta\]
so that
\[|\Theta_c(\theta)| = \sqrt{1-2c(1-c)(1-\cos\theta)}.\]
The estimate $\cos x \geq 1 - x^2/2$ yields 
\begin{equation}
\label{eq:yep}
|\Theta_c(\theta)| \geq \sqrt{1-c(1-c)\theta^2} \geq1-c(1-c)\theta^2.
\end{equation}
for $\theta$ sufficiently small. In lemma \ref{thm:closedarc} it was proved that if $z=\chi(g)/\chi(1)$ satisfies $\re z^k \leq 0$ for some $k>B$ ($B$ the bound of Theorem \ref{thm:Cnopen}), then $|z|\leq 1-1/k^2$. Since $\theta_j\to 0$ and $\theta_j\neq 0$ there exists $k>B$ and angle $\theta_j$ of $z_j$ for which
\begin{equation}
\label{equ:smallangle}
\frac{\pi}{2(k+1)}\leq |\theta_j| < \frac{\pi}{2k}
\end{equation}
So that $|z_j|\leq 1-1/k^2$. By assumption, $|z_j|>|\Theta_{c_j}(\theta_j)|$ and the estimates \eqref{eq:yep}, \eqref{equ:smallangle} yield the comparison
\[1-\frac{1}{k^2}>1-c_j(1-c_j)\left(\frac{\pi}{2k}\right)^2,\] or
\[\frac{4}{\pi^2}<c_j(1-c_j),\] contradicting $c(1-c)\leq1/4$ for all $0<c<1$.
\end{proof}
\bibliographystyle{abbrv}
\bibliography{researchstatementbib}

\end{document}